\documentclass[11pt]{article}

\usepackage[a4paper,textwidth=15cm,textheight=25cm]{geometry}
\usepackage{amssymb}
\usepackage{amsmath}
\usepackage{amsfonts}
\usepackage{bm}
\usepackage{authblk}
\usepackage{appendix}

\newtheorem{theorem}{Theorem}
\newtheorem{corollary}[theorem]{Corollary}
\newtheorem{definition}[theorem]{Definition}
\newtheorem{lemma}[theorem]{Lemma}
\newtheorem{proposition}[theorem]{Proposition}
\newtheorem{remark}[theorem]{Remark}
\newtheorem{example}[theorem]{Example}
\newenvironment{proof}[1][Proof]{\noindent\textbf{#1. }}{\hfill \rule{0.5em}{0.5em} \medskip\newline }

\newcommand{\dE}{\mathbb{E}}

\newcommand{\TXs}[1]{\overline{T}_{X,#1}}
\newcommand{\Ts}[2]{\overline{T}_{#1,#2}}

\newcommand{\TYs}[1]{\overline{T}_{Y,#1}}
\newcommand{\muXs}[1]{\widetilde{\mu}_{X,#1}}

\newcommand{\sFR}[1]{\ensuremath{#1\!-\!}{\rm FR}}

\title{
On the limit behavior of iterated equilibrium distributions for the Gamma and Weibull families\thanks{This work was partially supported by the Centre for Mathematics of the University of Coimbra -- UID/MAT/00324/2013, funded by the Portuguese Government through FCT/MEC and co-funded by the European Regional Development Fund through the Partnership Agreement PT2020.}}

\author[1]{Idir Arab}
\author[2]{Milto Hadjikyriakou}
\author[1]{Paulo Eduardo Oliveira\footnote{Corresponding author.}}

\affil[1]{CMUC, Department of Mathematics, University of Coimbra, Portugal}
\affil[2]{University of Central Lancashire, Cyprus}

\date{}
\begin{document}

\maketitle

\begin{abstract}
In this paper, we study the evolution of iterated equilibrium distributions for the Gamma and Weibull families of distributions as the iteration step increases. 
We characterize their moments and the pointwise limit of the distribution functions corresponding to the iterated distributions. As a byproduct, we obtain approximations for higher order moments of the residual lifetime.

\medskip


\medskip


\medskip

\end{abstract}

\section{Introduction}
Iterated distributions were introduced by Averous and Meste~\cite{AM89}, in order to construct a classification of lifetime distributions, and were initially studied in a more systematic way by Fagiuoli and Pellerey~\cite{FP93}. The iteration procedure described below produces what is also known as equilibrium distributions in economics or actuarial activities, since they describe the distribution of the first drop below the initial reserve. Moreover, equilibrium distributions play an important role in ageing relations (see, for example, Chatterjee and Mukherjee~\cite{CM01}) or in renewal theory (see Cox~\cite{Cox62}). A suitable representation of the iterated distributions describes their tails as normalized moments of stop-loss risk premiums with a given deductible, a common type of contract in actuarial activity, with interest to the characterization of ruin probabilities and insolvency. We obtain a characterization of the asymptotics of these moments, with respect to their order, for initial random variables with distribution in the Gamma or Weibull families. These results provide simple numerical approximations for the stop-loss premiums.

Let us describe our framework and introduce the basic notation and transformations. We will be assuming throughout that $X$ is a nonnegative random variable with density function $f_X$, distribution function $F_X$, and tail function $\overline{F}_X=1-F_X$.
\begin{definition}
\label{def:s-iter}
For each $x\geq 0$, define
\begin{equation}
\TXs{0}(x)=f_{X}(x)\quad \mbox{and}\quad \muXs{0}=\int_0^\infty \TXs{0}(t)\,dt=1.
\end{equation}
For each $s\geq 1$, define the $s-$iterated distribution induced by $X$, $T_{X,s}$ by its tail $\TXs{s}=1-T_{X,s}$ as follows:
\begin{equation}
\label{eq:s-iter}
\TXs{s}(x)=\frac{1}{\muXs{s-1}}\int_{x}^\infty\TXs{s-1}(t)\,dt
\quad \mbox{where}\quad \muXs{s}=\int_0^\infty \TXs{s}(t)\,dt,
\end{equation}
assuming the integrals above are finite.
\end{definition}
The distribution $\TXs{2}$ is known as the equilibrium distribution of $X$. Hence, the iteration process above defines, for each $s\geq 1$, the equilibrium distribution $\TXs{s}$ of a random variable with tail $\TXs{s-1}$. Thus, the iteration procedure given in Definition~\ref{def:s-iter} may be restated in terms of the iterated equilibrium distribution.
\begin{definition}
The $(s-1)-$iterated equilibrium distribution of $X$ has tail given by $\TXs{s}$, that is, the $s-$iterated distribution induced by $X$ described in Definition~\ref{def:s-iter}.
\end{definition}
Taking into account this identification, we shall refer throughout to iterated distributions instead of iterated equilibrium, as these only differ on the count of the iteration steps.

It is easily verified that the $s-$iterated distribution induced by an exponential random variable $X$ is exactly the same exponential distribution as $X$. That is, the exponential distributions are fixed points with respect to the iteration procedure introduced in Definition~\ref{def:s-iter}. Moreover, it is also easily verified that explicit identification of the iterated distributions may become in general quite complex once we leave the case of $X$ being exponentially distributed.

As already mentioned, iterated distributions were used by Averous and Meste~\cite{AM89} to classify the distribution of $X$ with respect to its tail behavior. General properties of the iterated distributions were studied in Fagiuoli and Pellerey~\cite{FP93} and Nanda et al.~\cite{ASOK}, but to the best of our knowledge, the behavior of the iterated distributions with respect to the iteration step has not been studied before. In this paper, we will be mainly interested on this behavior of iterated distributions as the iteration step grows.

Although the iterated distributions are defined in a recursive way, the following theorem gives a useful closed form representation.
\begin{theorem}[Lemma 2 and Remark 3 in Arab and Oliveira~\cite{AO18}]
\label{Simple T_s}
Assume $X$ is an absolutely continuous nonnegative random variable with finite moment of order $s-1$ for $s=1,2,\ldots$. Then the $s-$iterated tail $\TXs{s}$ may be represented as
\begin{equation}
\label{eq:simpleTs}
\TXs{s}(x)=\frac{1}{\dE X^{s-1}}\int_x^\infty f_{X}(t)(t-x)^{s-1}\,dt=\frac{1}{\dE X^{s-1}}\dE(X-x)_+^{s-1},
\end{equation}
where $(X-x)_+=\max(0,X-x)$ is the residual lifetime at age $x$. 
\end{theorem}
\begin{remark}
It follows from (\ref{eq:simpleTs}) that the $s-$iterated distribution may be interpreted as a normalized survival moment of order $s-1$. It is also worth mentioning that (\ref{eq:simpleTs}) means that, up to a normalizing factor, the $s-$iterated distribution is the stop-loss transform of order $s-1$, $\dE(X-x)_+^{s-1}$, a common quantity of interest in actuarial models (see, for example, Cheng and Pai~\cite{CP03}, Nair et al.~\cite{NSS13}, Tsai~\cite{Tsai06} or Rachev and R\"{u}schendorf~\cite{RR90}, among many other references).
\end{remark}

Although Theorem~\ref{Simple T_s} provides a closed representation for $\TXs{s}$, it does not seem to be very helpful for actual calculations. An illustrative example, that we will be exploring in detail later, is obtained by assuming that $X$ has distribution in the Gamma or Weibull families, which are important classes of distributions in many different research fields such as reliability theory.



The paper is organized as follows: In Section 2, we provide a formula for the higher order moments of the iterated distributions. In Section 3, a recursive representation for the high order iterated distributions of a convolution is obtained and this result is used to find an expression for the iterated distributions of a random variable that is Gamma distributed with integer shape parameter. In Section 4, we derive an explicit expression for the $s-$iterated distribution of a Gamma distributed random variable, still assuming the shape parameter is an integer and furthermore, we study the limit behavior of the iterated distribution as the iteration step tends to infinity, now for general shape parameter. Finally in Section 5, we identify the limit behavior for the iterated distributions induced from the Weibull family of distributions.

\section{Moments of iterated distributions}
We shall start by a characterization of the moments of iterated distributions. First, remark that the normalizing constants $\muXs{s}$ in (\ref{eq:s-iter}) are obviously the mathematical expectation of the $s-$iterated distributions. These first order moments have been characterized, in terms of moments of the inducing random variable, in Corollary~2.1 by Nanda et al.~\cite{NJS96}:
\begin{equation}
\label{eq:moms}
\muXs{s}=\frac{1}{s}\frac{\dE X^{s}}{\dE X^{s-1}}.
\end{equation}

The following result extends (\ref{eq:moms}), characterizing higher order moments of the iterated distributions.
\begin{proposition}
Assume $X$ is an absolutely continuous nonnegative random variable with finite moment of order $m+s-1$, where $m\geq 1$ and $s\geq 2$. Then the $s-$iterated distribution induced by $X$ has a finite moment of order $m$ given by
$$
\mu_{s,m}={m+s-1\choose m}^{-1}\,\frac{\dE X^{m+s-1}}{\dE X^{s-1}}.
$$
\end{proposition}
\begin{proof}
It follows from (\ref{eq:s-iter}) that the $s-$iterated distribution has density $\frac{1}{\muXs{s-1}}\TXs{s-1}$. So, using the first representation for $\TXs{s-1}$ in (\ref{eq:simpleTs}), we have
$$
\mu_{s,m}=\frac{s-1}{\dE X^{s-1}}\int_0^\infty\!\int_x^\infty x^m(t-x)^{s-2}f(t)\,dt\,dx.
$$
By inverting the integration order, we easily find that
$$
\mu_{s,m}=\frac{s-1}{\dE X^{s-1}}\int_0^\infty\!\int_0^t x^m(t-x)^{s-2}f(t)\,dx\,dt
=\frac{s-1}{\dE X^{s-1}}\frac{\Gamma(m+1)\Gamma(s-1)}{\Gamma(m+s)}\int_0^\infty t^{m+s-1}f(t)\,dt,
$$
where $\Gamma$ is the Euler function, which leads to the desired result.
\end{proof}
The following description for the variance of iterated distributions is now straightforward.
\begin{corollary}
Assume $X$ is an absolutely continuous nonnegative random variable with finite moment of order $s+1$. The variance of the $s-$iterated distribution induced by X is
$$
\sigma^2_s=\frac{1}{s}\frac{\dE X^s}{\dE X^{s-1}}
\left(\frac{2}{s+1}\frac{\dE X^{s+1}}{\dE X^s}-\frac{1}{s}\frac{\dE X^s}{\dE X^{s-1}}\right).
$$
\end{corollary}
As mentioned earlier, the exponential distributions are fixed points for the iteration procedure, so one could expect to have iterated distributions converging to the exponential. The above corollary shows that this may not happen. Indeed, it follows immediately that if the quotient $\frac{\dE X^s}{\dE X^{s-1}}$ is bounded with respect to $s\geq 1$, then $\lim_{s\rightarrow+\infty}\sigma^2_s=0$, implying that the equilibrium distributions are, in such cases, converging to a degenerate distribution. We will prove later that the Weibull distributions satisfy this asymptotic degeneracy.

\section{Iterated distribution of convolutions}
Summing independent random variables is a common way to introduce new families of distributions, appearing as convolution powers built upon some initial distribution. With this in mind, we shall derive a characterization for the iterated distribution for the $n$-th convolution power based on expressions for the $(n-1)$-th convolution power. For this purpose, it is convenient to describe the distribution in terms of their densities rather than using distribution functions. Of course, as follows from Definition~\ref{def:s-iter}, the density $f_s$ for the $s-$iterated distribution is, up to multiplication by a constant, the tail for the $(s-1)-$iterated distribution:
%
%
%
\begin{equation}
\label{eq:density}
f_s(x) = \frac{1}{\muXs{s-1}}\TXs{s-1}(x)= (s-1)\frac{E(X-x)_{+}^{s-2}}{EX^{s-1}}=\frac{(s-1)}{EX^{s-1}}\int_{x}^{\infty}(t-x)^{s-2}f(t)\,dt.
\end{equation}
In the sequel, we will be representing the convolution of density functions by $\ast$ defined as follows:
\[
f\ast g (x) =\int_{0}^{x}f(t)g(x-t)\,dt.
\]
Moreover, the expression $f^{n\ast}$ will represent the $n$-th convolution power of a density function $f$.
\begin{theorem}
\label{thm:conv}
Let $X_1,\ldots,X_n$ be nonnegative i.i.d random variables with the same distribution as $X$ and define $S_n =X_1+\cdots+X_n$. For every $s,n\geq2$, the density of the $s-$iterated distribution induced by $S_n$ is given by
\begin{equation}
\label{eq:conv}
f_s^{n*}(x)=\frac{\mu^{(n-1)*}_{s-1}}{\mu^{n*}_{s-1}}f\ast f_{s}^{(n-1)*}(x)+\frac{1}{\mu^{n*}_{s-1}}\sum_{\ell=1}^{s-1}{{s-1}\choose {\ell}}\mu^{(n-1)*}_{s-\ell-1}\mu_{\ell}f_{\ell+1}(x),
\end{equation}
where $\mu^{n*}_s= \dE S_n^s$ and $\mu_s =\dE X^s$.
\end{theorem}
\begin{proof}
Using (\ref{eq:density}) for the $n$-th convolution power and reversing the integration order, we find that
\begin{eqnarray}
\nonumber f_{s}^{n*}(x)&=&\frac{(s-1)}{\mu^{n*}_{s-1}}\int_{x}^{\infty}(t-x)^{s-2}f^{n*}(t)\,dt\\
&=& \nonumber \frac{(s-1)}{\mu^{n*}_{s-1}}\int_{x}^{\infty}(t-x)^{s-2}\left(\int_{0}^{t}f(u)f^{(n-1)*}(t-u)\,du\right)\,dt\\
&=&\nonumber \frac{(s-1)}{\mu^{n*}_{s-1}}\int_{0}^{x}f(u)\left(\int_{x}^{\infty}f^{(n-1)*}(t-u)(t-x)^{s-2}\,dt\right)\,du\\
& & \qquad+\nonumber \frac{(s-1)}{\mu^{n*}_{s-1}}\int_{x}^{\infty}f(u)\left(\int_{u}^{\infty}f^{(n-1)*}(t-u)(t-x)^{s-2}\,dt\right)\,du\\
&=&\nonumber\frac{(s-1)}{\mu^{n*}_{s-1}}\int_{0}^{x}f(u)\left(\int_{x-u}^{\infty}f^{(n-1)*}(v)(v+u-x)^{s-2}\,dv\right)\,du\\
& & \qquad+ \nonumber \frac{(s-1)}{\mu^{n*}_{s-1}}\int_{x}^{\infty}f(u)\left(\int_{0}^{\infty}f^{(n-1)*}(t)(t+u-x)^{s-2}\,dt\right)\,du\\
&=&\label{int} I_1+I_2.
\end{eqnarray}
We rewrite $I_1$ as
\begin{equation}
\label{int1}
I_1=\frac{(s-1)}{\mu^{n*}_{s-1}}\int_{0}^{x}f(u)\frac{\mu^{(n-1)*}_{s-1}}{(s-1)}f_{s}^{(n-1)*}(x-u)\,du
=\frac{\mu^{(n-1)*}_{s-1}}{\mu^{n*}_{s-1}}f\ast f_{s}^{(n-1)*}(x).
\end{equation}
As what regards $I_2$, we have that
\begin{eqnarray}
\nonumber I_2&=&
   \frac{(s-1)}{\mu^{n*}_{s-1}}\int_{x}^{\infty}f(u)\left(\int_{0}^{\infty}f^{(n-1)*}(t)(t+u-x)^{s-2}\,dt\right)\,du\\
\nonumber &=&
   \frac{(s-1)}{\mu^{n*}_{s-1}}\int_{x}^{\infty}f(u)\left(\sum_{\ell=0}^{s-2}{{s-2}\choose{\ell}}(u-x)^{\ell}\int_{0}^{\infty}t^{s-2-\ell}f^{(n-1)*}(t)\,dt\right)\,du\\
\nonumber &=&
   \frac{(s-1)}{\mu^{n*}_{s-1}}\int_{x}^{\infty}f(u)\left(\sum_{\ell=0}^{s-2}{{s-2}\choose{\ell}}(u-x)^{\ell}\mu^{(n-1)*}_{s-2-\ell}\right)\,du\\
\nonumber &=&
   \frac{(s-1)}{\mu^{n*}_{s-1}}\sum_{\ell=0}^{s-2}{{s-2}\choose{\ell}}\mu^{(n-1)*}_{s-2-\ell}\int_{x}^{\infty}(u-x)^\ell f(u)\,du\\
\nonumber &=&
  \frac{(s-1)}{\mu^{n*}_{s-1}}\sum_{\ell=0}^{s-2}{{s-2}\choose{\ell}}\mu^{(n-1)*}_{s-2-\ell}\,\frac{\mu_{\ell+1}}{\ell+1}f_{\ell+2}(x)\\
 &=&\label{int2}\frac{1}{\mu^{n*}_{s-1}}\sum_{\ell=1}^{s-1}{{s-1}\choose{\ell}}\mu^{(n-1)*}_{s-1-\ell}\mu_{\ell}f_{\ell+1}(x).
\end{eqnarray}
The conclusion (\ref{eq:conv}) now follows by combining (\ref{int}), (\ref{int1}) and (\ref{int2}).
\end{proof}
The result above provides a recursive formula for the high order iterated distributions of the convolution $f^{n*}(x)$. A simple application is given next, identifying explicitly the iterated distributions induced by integer shape parameter Gamma distributions.

\begin{example}
The $\Gamma(2,\lambda)$ is the 2-nd convolution power of the exponential distribution with hazard rate $\lambda$, 
whose density function is $f(x) =\lambda e^{-\lambda x}$, for $x\geq 0$. 
With the notation introduced above, the density function $g(x)$ of the $\Gamma(2,\lambda)$ distribution is represented as $g(x)=f^{2\ast}(x)$, hence the $s-$iterated distribution induced by the $\Gamma(2,\lambda)$ distribution has density $g_s(x)=f_s^{2\ast}(x)$, and we may use (\ref{eq:conv}) to obtain a recursive representation.

As the exponential is a fixed point for the iterative procedure introduced in Definition~\ref{def:s-iter}, we have $f_{s}^{1*}(x) = f_s(x)=f(x)=\lambda e^{-\lambda x}$, therefore $\mu_{s-1}^{1*}$ is the moment of order $s-1$ of an exponential random variable with hazard rate $\lambda$, that is, $\mu_{s-1}^{1*}=\frac{(s-1)!}{\lambda^{s-1}}$. Moreover, as the 2-nd convolution power is the $\Gamma(2,\lambda)$, we also know that $\mu_{s-1}^{2*}= \frac{\Gamma(s+1)}{\lambda^{s-1}}=\frac{s!}{\lambda^{s-1}}$. Furthermore,
$$
\int_{0}^{x}f(t)f^{1*}_s(x-t)\,dt =\int_{0}^{x}\lambda e^{-\lambda t}\lambda e^{-\lambda (x-t)}\,dt
  =\lambda^2\int_{0}^{x}e^{-\lambda x}\,dt=\lambda^2 xe^{-\lambda x}
$$
and
\begin{eqnarray*} \frac{1}{\mu^{n*}_{s-1}}\sum_{\ell=1}^{s-1}{{s-1}\choose{\ell}}\mu^{(n-1)*}_{s-\ell-1}\mu_{\ell}f_{\ell+1}(x)
 & = &
  \frac{\lambda^{s-1}}{s!}\sum_{\ell=1}^{s-1}{{s-1}\choose{\ell}}\frac{(s-\ell-1)!}{\lambda^{s-\ell-1}}\frac{\ell!}{\lambda^\ell}\lambda e^{-\lambda x}\\
 & = &\frac{\lambda (s-1)}{s}e^{-\lambda x}.
\end{eqnarray*}
Hence, the $s-$iterated distribution induced by the $\Gamma(2,\lambda)$ distribution has density
\begin{equation}
\label{pdf gamma}g_s(x)=f_s^{2*}(x) = \frac{1}{s}\lambda^2xe^{-\lambda x}+\frac{s-1}{s}f(x).
\end{equation}
\end{example}

This approach may be extended to general Gamma distributions with integer valued shape parameter. The argument goes along the same line, needing some extra effort on the manipulation of some combinatorial sums. We first state an auxiliary combinatorial lemma that is useful for obtaining the results that follow.

\begin{lemma}
\label{columns formula}
For every $m,k\geq1$, we have that
\[
\sum_{j=0}^{m}\binom{k+j}{k}=\binom{k+m+1}{m}.
\]
\end{lemma}
\begin{proof}
This follows easily by induction on $m$.
\end{proof}
\begin{proposition}
\label{prop:sGamma}
Let $X$ be a random variable with distribution $\Gamma(n,\lambda)$, and $f(x)=\lambda e^{-\lambda x}$ be the density of an exponential distribution with hazard rate $\lambda>0$. For every $s,n\geq 2$, the density function of the $s-$iterated distribution induced by $X$ is
\begin{equation}
\label{eq:convGamma}g_s(x) = f_s^{n*}(x) = \frac{n-1}{n+s-2}f\ast f_s^{(n-1)*}(x)+\frac{s-1}{n+s-2}f(x).
\end{equation}
\end{proposition}
\begin{proof}
As $g_s(x) = f_s^{n*}(x)$, we will use the recursive representation from Theorem~\ref{thm:conv}. Note that the moments of convolutions $\mu_{s-1}^{n*}$ are just the moments of order $s-1$ of the $\Gamma(n,\lambda)$ distribution. Hence
\begin{equation}
\label{ratiomus}
\mu_{s-1}^{n*} =\frac{1}{\lambda^{s-1}}\frac{(n+s-2)!}{(n-1)!}
\qquad\mbox{and}\qquad
\frac{\mu_{s-1}^{(n-1)*}}{\mu_{s-1}^{n*}}=\frac{n-1}{n+s-2}.
\end{equation}
Furthermore,
\begin{eqnarray}
\nonumber\lefteqn{\frac{1}{\mu^{n*}_{s-1}}\sum_{\ell=1}^{s-1}{{s-1}\choose {\ell}}\mu^{(n-1)*}_{s-\ell-1}\mu_{\ell}f_{\ell+1}(x) } \\
\nonumber& & =\frac{(n-1)!\lambda^{s-1}}{(n+s-2)!}\sum_{\ell=1}^{s-1}{{s-1}\choose {\ell}}\frac{\Gamma(n+s-\ell-2)}{\Gamma(n-1)\lambda^{s-\ell-1}}\frac{\ell!}{\lambda^\ell}\lambda e^{-\lambda x}\\
\nonumber& & =(n-1)\lambda e^{-\lambda x}\sum_{\ell=1}^{s-1}{{s-1}\choose {\ell}}\ell!\frac{(n+s-\ell-3)!}{(n+s-2)!}\\
\nonumber& & =(n-1)\frac{\lambda e^{-\lambda x}}{n+s-2}\sum_{\ell=1}^{s-1}\frac{(s-1)!}{(s-1-\ell)!}\frac{(n+s-(\ell+3))!}{(n+s-3)!}\\
\nonumber& & =\frac{(n-1)\lambda e^{-\lambda x}}{(n+s-2)}\frac{1}{{{n+s-3}\choose{n-2}}}\sum_{\ell=1}^{s-1}{{n+s-\ell-3}\choose{n-2}}\\
 & &\label{term2}=\frac{(s-1)}{(n+s-2)}f(x),
\end{eqnarray}
where the last equality follows by applying Lemma~\ref{columns formula}. The proof is concluded by rewriting (\ref{eq:conv}) taking into account (\ref{ratiomus}) and (\ref{term2}).
\end{proof}

The representation just proved in Proposition~\ref{prop:sGamma} allows for a first characterization of the behavior of the iterated distributions as the iteration step goes to $+\infty$.

\begin{corollary}
\label{cor:Gammalim1}
Let $X$ be a random variable with distribution $\Gamma(n,\lambda)$. Then, for each $x\geq 0$ fixed, $\lim_{s\rightarrow+\infty} g_{s}(x) = f(x)$,
where $f(x)$ is the density of an exponential distribution with hazard rate $\lambda$.
\end{corollary}
\begin{proof}
The proof follows easily by observing that the convolution that appears in the first term on the right hand side of (\ref{eq:convGamma}) is bounded by $\lambda$. Allowing $s$ to tend to $+\infty$ we have the desired result.
\end{proof}

\noindent One could prefer to have a recursive characterization for the density of the $s-$iterated distribution involving only elementary operations, that is, avoiding convolutions on the recursive expression. This can be obtained by iteratively using (\ref{eq:conv}) to describe the $(n-1)$-th convolution power that appears in the right-hand side of (\ref{eq:conv}). We apply this technique for random variables that are Gamma distributed and this leads to a long algebraic manipulation proving the representation given below. Note that the iterated distributions induced by the Gamma family will be studied with a more explicit approach in the next section.

\begin{proposition}
\label{prop:GammaWOcon}
Let $X$ be a random variable with distribution $\Gamma(n,\lambda)$, with $\lambda>0$, and let $n\geq 2$. Then,
\begin{eqnarray*}
\mu^{n*}_{s-1} f_{s}^{n*}(x)-\mu^{(n-1)*}_{s-1} f_{s}^{(n-1)*}(x) &=&
\frac{(s-1)!}{\lambda^{s-1}}\left(\frac{\lambda^{n-1}x^{n-2}}{(n-2)!}e^{-\lambda x}\left(\frac{\lambda x}{n-1}-1\right)+1\right.\\
& &\quad\left. -{n+s-4\choose s-2}+e^{-\lambda x}\sum_{k=2}^{n-2}{s+k-2\choose k} \frac{\lambda^{n-k}x^{n-k-1}}{(n-k-1)!}\right).
\end{eqnarray*}
\end{proposition}
\begin{proof}
In order to obtain an alternative to the characterization in Proposition~\ref{prop:sGamma} without convolutions, let us pick up from the expression in Theorem~\ref{thm:conv} which can be written as
\begin{equation}
\label{eq:convn}
\mu^{n*}_{s-1} f_{s}^{n*}(x)=\mu^{(n-1)*}_{s-1}f\ast f_{s}^{(n-1)*}(x)+\sum_{\ell=1}^{s-1}{{s-1}\choose {\ell}}\mu^{(n-1)*}_{s-\ell-1}\mu_{\ell}f_{\ell+1}(x).
\end{equation}
We may rewrite this for the $(n-1)$-fold convolution, to get
$$
\mu^{(n-1)*}_{s-1} f_{s}^{(n-1)*}(x)=\mu^{(n-2)*}_{s-1}f\ast f_{s}^{(n-2)*}(x)+\sum_{\ell=1}^{s-1}{{s-1}\choose {\ell}}\mu^{(n-2)*}_{s-\ell-1}\mu_{\ell}f_{\ell+1}(x).
$$
The first term on the right in (\ref{eq:convn}) may be written as $f\ast (\mu^{(n-1)*}_{s-1}f_{s}^{(n-1)*})(x)$, so it can be replaced by the previous expression to find, after a rearrangement of the terms:
$$
\mu^{n*}_{s-1} f_{s}^{n*}(x)=\mu^{(n-2)*}_{s-1}f^{2\ast}\ast f_{s}^{(n-2)*}(x)+\sum_{\ell=1}^{s-1}{{s-1}\choose {\ell}}\mu_{\ell}\left(\mu^{(n-2)*}_{s-\ell-1}f\ast f_{\ell+1}(x)+\mu^{(n-1)*}_{s-\ell-1}f_{\ell+1}(x)\right).
$$
Again, we apply (\ref{eq:convn}) to the first term on the right-hand side above. So, iterating this substitution, we finally get the representation:
\begin{equation}
\label{eq:convn1}
\begin{array}{l}
\displaystyle\mu^{n*}_{s-1} f_{s}^{n*}(x)=\mu_{s-1}f^{(n-1)\ast}\ast f_{s}(x) \\
\displaystyle\quad+\sum_{\ell=1}^{s-1}{{s-1}\choose {\ell}}\mu_{\ell}\left(\mu_{s-\ell-1}f^{(n-2)*}(x)+\mu^{2*}_{s-\ell-1}f^{(n-3)*}(x)+\cdots+
\mu^{(n-2)\ast}_{s-\ell-1}f(x)+\mu^{(n-1)*}_{s-\ell-1}\right)\ast f_{\ell+1}(x).
\end{array}
\end{equation}
We may, of course, rewrite the representation above for $(n-1)$-fold convolution:
\begin{equation}
\label{eq:convn-1}
\begin{array}{l}
\displaystyle\mu^{(n-1)*}_{s-1} f_{s}^{(n-1)*}(x)=\mu_{s-1}f^{(n-2)\ast}\ast f_{s}(x) \\
\displaystyle\quad+\sum_{\ell=1}^{s-1}{{s-1}\choose {\ell}}\mu_{\ell}\left(\mu_{s-\ell-1}f^{(n-3)*}(x)+\mu^{2*}_{s-\ell-1}f^{(n-4)*}(x)+\cdots+
\mu^{(n-3)\ast}_{s-\ell-1}f(x)+\mu^{(n-2)*}_{s-\ell-1}\right)\ast f_{\ell+1}(x),
\end{array}
\end{equation}
from which follows that
\begin{eqnarray*}
	\lefteqn{\mu^{n*}_{s-1} f_{s}^{n*}(x)-\mu^{(n-1)*}_{s-1} f_{s}^{(n-1)*}(x)} \\
	& & \quad = \mu_{s-1}\left(f^{(n-1)\ast}-f^{(n-2)\ast}\right)\ast f_s(x) \\
	& &\qquad
	+\sum_{\ell=1}^{s-1}{{s-1}\choose {\ell}}\mu_\ell\left(\mu_{s-\ell-1}f^{(n-2)\ast}-\mu^{(n-2)\ast}_{s-\ell-1}\right)\ast f_{\ell+1}(x) \\
	& &\qquad
	+\sum_{\ell=1}^{s-1}{{s-1}\choose {\ell}}\mu_\ell\left(\sum_{k=2}^{n-2}\left(\mu^{k\ast}_{s-\ell-1}-\mu^{(k-1)\ast}_{s-\ell-1} \right)f^{(n-k-1)\ast} \right)\ast f_{\ell+1}(x).
\end{eqnarray*}
\noindent
Since $f$ is the exponential density, the iterated density $f_\ell=f$, hence, we may rewrite:
\begin{eqnarray*}
\lefteqn{\mu^{n*}_{s-1} f_{s}^{n*}(x)-\mu^{(n-1)*}_{s-1} f_{s}^{(n-1)*}(x)} \\
	& & \quad = \mu_{s-1}\left(f^{n\ast}-f^{(n-1)\ast}\right) \\
	& &\qquad
	+\sum_{\ell=1}^{s-1}{{s-1}\choose {\ell}}\mu_\ell\left(\mu_{s-\ell-1}f^{(n-1)\ast}-\mu^{(n-2)\ast}_{s-\ell-1}f\right)\\
	& &\qquad
	+\sum_{\ell=1}^{s-1}{{s-1}\choose {\ell}}\mu_\ell\left(\sum_{k=2}^{n-2}\left(\mu^{k\ast}_{s-\ell-1}-\mu^{(k-1)\ast}_{s-\ell-1} \right)f^{(n-k)\ast} \right)\\
	& & = A_1+A_2+A_3.
\end{eqnarray*}
We may now compute each of these three terms. Recalling that the convolution of exponentials is a gamma density, and the expressions for the moments, it follows easily that
$$A_1=\frac{(s-1)!}{\lambda^{s-1}}\frac{\lambda^{n-1}x^{n-2}}{(n-2)!}e^{-\lambda x}\left(\frac{\lambda x}{n-1}-1\right).
$$
Recall that
$$
\begin{array}{rcl}
A_2 & = &
\displaystyle\sum_{\ell=1}^{s-1}{{s-1}\choose {\ell}}\mu_\ell\left(\mu_{s-\ell-1}f^{(n-1)\ast}-\mu^{(n-2)\ast}_{s-\ell-1}f\right)\\
& = & \displaystyle\sum_{\ell=1}^{s-1}{{s-1}\choose {\ell}}\frac{\ell!}{\lambda^{s-1}}\left((s-\ell-1)! f^{(n-1)\ast}-\frac{(n+s-\ell-4)!}{(n-3)!}f\right).
\end{array}
$$
So, we need to compute two summations:
$$
\sum_{\ell=1}^{s-1}{{s-1}\choose {\ell}}\frac{\ell!}{\lambda^{s-1}}(s-\ell-1)!=\frac{(s-1)!}{\lambda^{s-1}}.
$$
and, using Lemma~\ref{columns formula},
$$
\sum_{\ell=1}^{s-1}{{s-1}\choose {\ell}}\frac{\ell!}{\lambda^{s-1}}\frac{(n+s-\ell-4)!}{(n-3)!}
=\frac{(s-1)!}{\lambda^{s-1}}\sum_{\ell=1}^{s-1}{n+s-\ell-4\choose n-3}=\frac{(s-1)!}{\lambda^{s-1}}{n+s-4\choose s-2}.
$$
\noindent
We now compute
$$
A_3=\sum_{\ell=1}^{s-1}{{s-1}\choose {\ell}}\mu_\ell\left(\sum_{k=2}^{n-2}\left(\mu^{k\ast}_{s-\ell-1}-\mu^{(k-1)\ast}_{s-\ell-1} \right)f^{(n-k)\ast} \right).
$$
Note first that $\mu^{k\ast}_0-\mu^{(k-1)\ast}_0=0$, so the summation ranges only from $\ell=1$ up to $s-2$. Further, for $\ell>0$,
$$
\mu^{k\ast}_\ell-\mu^{(k-1)\ast}_\ell=\frac{\ell}{\lambda^\ell}\frac{(k+\ell-2)!}{(k-1)!}.
$$
Replacing these expressions and inverting the summations, we find
$$
\begin{array}{rcl}
A_3 & = & \displaystyle \sum_{\ell=1}^{s-2}\frac{(s-1)!}{\ell!(s-\ell-1)!}\frac{\ell!}{\lambda^\ell}
\left(\sum_{k=2}^{n-2}\left(\mu^{k\ast}_{s-\ell-1}-\mu^{(k-1)\ast}_{s-\ell-1} \right)f^{(n-k)\ast} \right)\\
& = &\displaystyle
\sum_{k=2}^{n-2}\sum_{\ell=1}^{s-2}\frac{(s-1)!}{(s-\ell-2)!}\frac{1}{\lambda^{s-1}} \frac{(s+k-\ell-3)!}{(k-1)!}f^{(n-k)\ast} \\
& = & \displaystyle
\frac{(s-1)!}{\lambda^{s-1}}\sum_{k=2}^{n-2}\sum_{\ell=1}^{s-2} {s+k-\ell-3\choose k-1} f^{(n-k)\ast} \\
& = & \displaystyle
\frac{(s-1)!}{\lambda^{s-1}}\sum_{k=2}^{n-2}{s+k-2\choose k}, 
\end{array}
$$
using Lemma~\ref{columns formula} for the final equality.
\end{proof}


\section{Iterated distributions induced by Gamma distributions}
Given a random variable $X$, the iterated distributions induced by $\frac{X}{\theta}$, where $\theta>0$, are easily related to the iterated distributions induced by $X$. Indeed, it follows immediately from (\ref{eq:simpleTs}) that, for every $s\geq 1$ and $\theta>0$, $\Ts{\frac{X}{\theta}}{s}(x)=\TXs{s}(\frac{x}{\theta})$. Therefore, in order to characterize the iterated distributions induced by a Gamma distributed random variable, it is enough to treat the case where $X$ has distribution $\Gamma(\alpha,1)$. We first derive an explicit expression for the $s-$iterated distribution when the shape parameter $\alpha$ is an integer.

\begin{theorem}
\label{thm:gamma}
Assume $X$ is $\Gamma(\alpha,1)$ distributed with integer shape parameter $\alpha\geq 2$. For every $s\geq 2$, the $s-$iterated distribution induced by $X$ has tail given by
\begin{equation}
\label{eq:tailGamma-s}
\TXs{s}(x) = e^{-x }+ \frac{e^{-x}}{\binom{\alpha+s-2}{\alpha-1}}\sum_{\ell=1}^{\alpha-1} \binom{s+\alpha-\ell-2}{\alpha-\ell-1}\frac{x^\ell}{\ell !}.
\end{equation}
\end{theorem}
\begin{proof}
We start by calculating the stop-loss transform of order $s$ for the random variable $X$.
\begin{eqnarray*}
\dE(X-x)^s_+&=&\frac{1}{\Gamma(\alpha)}\int_{x}^{\infty}(t-x)^st^{\alpha-1}e^{-t}dt\\
 &=&\frac{e^{-x}}{\Gamma(\alpha)}\int_{0}^{\infty}u^s(u+x)^{\alpha-1}e^{-u}du\\
 &=&e^{-x}\frac{(s+\alpha-1)!}{(\alpha-1)!}+e^{-x}\sum_{k=1}^{\alpha-1}\frac{(s+\alpha-1-k)!}{(\alpha-1-k)!}\frac{x^k}{k!},
\end{eqnarray*}
where the last equality follows by applying the binomial expansion. The result follows by writing
\[
\TXs{s}(x) = \frac{1}{\dE X^{s-1}}\dE(X-x)^{s-1}_+=\frac{(\alpha-1)!}{(\alpha+s-2)!}\dE(X-x)^{s-1}_{+}.
\]
\end{proof}
Note that the representation (\ref{eq:tailGamma-s}) can also be derived using (\ref{eq:convGamma}) recursively. However, the argument used above makes the derivation much simpler.

The formula obtained in (\ref{eq:tailGamma-s}) allows for an easy and direct verification that it is compatible with the result of Corollary~\ref{cor:Gammalim1}, that is, that $\TXs{s}$ converges pointwise to the survival function of the exponential distribution as $s\longrightarrow+\infty$. However, as in Corollary~\ref{cor:Gammalim1}, this convergence is only established for the case where the shape parameter is an integer.

The remaining part of this section extends the characterization of the pointwise limit behavior of the Gamma distribution, to non-integer shape parameter.
To this end, we need a way to control the iterated tails with respect to the variation of the shape parameter. For this purpose, we recall a definition introduced by Fagiuoli and Pellerey~\cite{FP93}.
\begin{definition}
Let $X$ and $Y$ be nonnegative valued random variables and $s\geq 1$. The random variable $Y$ is said to be larger than $X$ in $\sFR{s}$ ordering ($X\leq_{\sFR{s}}Y$) if $\frac{\TYs{s}}{\TXs{s}}$ is nondecreasing in $x\geq0$.
\end{definition}
The definition given above introduces an order relation between nonnegative random variables that, according to the result that follows, is hereditary with respect to the iteration parameter. Note that although the result is known, we were not able to find its proof in any published paper. Therefore, the proof is provided here for the sake of completeness.
\begin{lemma}[Theorem~3.4 in Fagiuoli and Pellerey~\cite{FP93}]
\label{SFR}
Assume $X$ and $Y$ are nonnegative valued random variables and $s\geq 1$. If $X\leq_{\sFR{s}}Y$, then $X\leq_{\sFR{(s+1)}}Y$.
\end{lemma}
\begin{proof}
Let $a$ be a positive real number. In order to prove that $\frac{\TYs{s+1}}{\TXs{s+1}}$ is increasing, it is enough to prove that $H(x)=\TYs{s+1}(x)-a\TXs{s+1}(x)$ changes sign at most once in the order ``$-,+$'', when $x$ traverses from 0 to $+\infty$. It is easily verified that $H^{\prime}(x)$ has the same sign variation as $V(x)=-(\TYs{s}(x)-b\TXs{s}(x))$,
where $b=a\frac{\mathbb{E}X^{s-1}\mathbb{E}Y^{s}}{\mathbb{E}X^{s}\mathbb{E}Y^{s-1}}$. Now, using the fact that $\frac{\TYs{s}}{\TXs{s}}$ is nondecreasing, the sign variation of $H^{\prime}$ is either ``$-$'', ``$+$'' or ``$+,-$''. Taking into account that $\lim_{x\rightarrow +\infty} H(x)=0$, we get that the sign variation of $H(x)$ is at most ``$-,+$''.
\end{proof}
Next, we establish the \sFR{s} order relationship within the Gamma family of distributions.
\begin{theorem}
\label{thm:sFR}
Let $X$ be a random variable with distribution $\Gamma(\alpha,1)$ and $Y$ be a random variable with distribution $\Gamma(\beta,1)$ where $\alpha,\beta >0$. Then $X\leq_{\sFR{s}}Y$ if and only if $\alpha<\beta$.
\end{theorem}
\begin{proof}
Taking into account Lemma~\ref{SFR}, it is enough to prove that $X\leq_{\sFR{1}}Y$, that is, that $\frac{\TYs{1}}{\TXs{1}}$ is nondecreasing. For this purpose, it is enough to prove that, for every $a>0$, $H(x)=\TYs{1}(x)-a\TXs{1}(x)$ changes sign at most once when $x$ goes from 0 to $+\infty$, and that if the change occurs it is in the order ``$-,+$''. Differentiating, we find that
$$
H^{\prime}(x)=a\frac{f_{X}(x)}{\dE X}-\frac{f_{Y}(x)}{\dE Y}
     =x^{\alpha-1}e^{-x}\left(\frac{a}{\Gamma(\alpha)}-\frac{x^{\beta-\alpha}}{\Gamma(\beta)}\right).
$$
As $\beta>\alpha$, the sign variation of $H^{\prime}(x)$ is ``$+,-$'', so the sign variation of $H(x)$ is either ``$+$'' or ``$-,+$'' which means that $H(x)$ is nondecreasing, hence $X\leq_{\sFR{1}}Y$.
\end{proof}
Before addressing the proof of the limit behavior of the iterated distributions, we quote here, with a suitable rephrasing, a result needed in course of the proof. Recall that for a random variable $X$, the failure rate of its distribution function $F_X$ at $x$ is defined as $f_X(x)/\overline{F}_X(x)$.
\begin{lemma}[Theorem~12 in Arab and Oliveira~\cite{AO18}]
\label{lem:thm12}
Let $X$ be a random variable with distribution $\Gamma(\alpha,\theta)$, where $\theta>0$, and $s\geq 1$ an integer. If $\alpha>1$ (resp., $\alpha <1$), then the failure rate of the $s-$iterated distribution induced by $X$ is increasing (resp., decreasing).
\end{lemma}
\begin{theorem}
Let $X$ be a random variable with distribution $\Gamma(\alpha,1)$, where $\alpha>0$. Then, for each $x\geq 0$ fixed, $\lim_{s\rightarrow+\infty}\TXs{s}(x)=e^{-x}$.
\end{theorem}
\begin{proof}
For the purpose of this proof, we shall denote by $\Ts{\alpha}{s}(x)$ the $s-$iterated distribution induced by a Gamma random variable with shape parameter $\alpha$. Hence, $\TXs{s}(x)=\Ts{\alpha}{s}(x)$. We separate the proof into three cases, depending on the value of $\alpha$.

\begin{description}
\item[{\rm\textit{The case $\alpha=1$}.}] This means that $X$ is exponentially distributed, which is a fixed point for the iteration procedure, so the conclusion is obvious.


\item[{\rm\textit{The case $\alpha>1$}.}]
If $\alpha$ is integer, the result follows from Corollary~\ref{cor:Gammalim1}. If $\alpha$ is not an integer, choose some integer $\beta>\alpha$, and consider two random variables $X_1$ and $X_2$ with distribution $\Gamma(1,1)$ and $\Gamma(\beta,1)$, respectively. Then Theorem~\ref{thm:sFR} implies that, for every $s\geq 1$, $X_1\leq_{\sFR{s}}X\leq_{\sFR{s}}X_2$, which further implies that
$$
e^{-x}=\Ts{1}{s}(x)\leq \Ts{\alpha}{s}(x)\leq \Ts{\beta}{s}(x).
$$
so, using again Corollary~\ref{cor:Gammalim1}, the proof is concluded.


\item[{\rm\textit{The case $\alpha<1$}.}] According to Lemma~\ref{lem:thm12}, the failure rate corresponding to the $s-$iterated distribution induced by $X$ is decreasing, which implies that $\Ts{\alpha}{s}(x)\leq\Ts{\alpha}{s+1}(x)$, for every $x\geq 0$ and $s\geq 1$. Moreover, from Theorem~\ref{thm:sFR}, it follows $\Ts{\alpha}{s}(x)\leq e^{-x}$, hence the limit $g(x)=\lim_{s\rightarrow+\infty}\Ts{\alpha}{s}(x)$ exists and $g(x)\leq e^{-x}$, for every $x\geq 0$. Now, we want to prove that the latter inequality is, indeed, an equality. For this purpose, define
%
\[
U_s(x) = \Ts{\alpha+1}{s}(x)-\Ts{m+1}{s}(x),
\]
where $m$ is an integer. By applying (\ref{eq:simpleTs}), we get that
$$
\Ts{\alpha+1}{s+1}(x) =  \frac{1}{\Gamma(\alpha+s+1)}\int_{x}^{\infty}(t-x)^st^{\alpha}e^{-t}\,dt
 = \frac{s}{\alpha+s}\Ts{\alpha+1}{s}(x)+\frac{\alpha}{\alpha+s}\Ts{\alpha}{s+1}(x).
$$
Hence,
\begin{eqnarray*}
\lefteqn{U_{s+1}-U_s} \\
 & &
 =\frac{s}{\alpha+s}\Ts{\alpha+1}{s}+\frac{\alpha}{\alpha+s}\Ts{\alpha}{s+1}- \left(\frac{s}{m+s}\Ts{m+1}{s}+\frac{m}{m+s}\Ts{m}{s+1}\right) - (\Ts{\alpha+1}{s}-\Ts{m+1}{s}) \\
 & &
 =\frac{m}{m+s}\left(\Ts{m+1}{s}-\Ts{m}{s+1}\right)+\frac{\alpha}{\alpha+s}\left(\Ts{\alpha}{s+1}-\Ts{\alpha+1}{s}\right).
\end{eqnarray*}
For $m=1$ and using (\ref{eq:tailGamma-s}), the latter expression becomes
\[
U_{s+1}(x)-U_s(x) = \frac{1}{s(s+1)}xe^{-x} + \frac{\alpha}{\alpha+s}\left(\Ts{\alpha}{s+1}(x)-\Ts{\alpha+1}{s}(x)\right).
\]
Recall that $\Ts{\alpha}{s+1}\leq g(x)$, while, based again on Lemma~\ref{lem:thm12}, we have that
$\Ts{\alpha+1}{s}\geq e^{-x}$. Therefore
\begin{equation}
\label{eq:final}
U_{s+1}(x)-U_s(x)\leq \frac{1}{s(s+1)}xe^{-x} - \frac{\alpha}{\alpha+s}(e^{-x}-g(x)).
\end{equation}
Assume now that $g(x)<e^{-x}$. Then, for each $x\geq 0$ fixed and $s$ (depending on $x$) large enough, the upper bound in (\ref{eq:final}) becomes negative, implying that $U_s(x)$ is decreasing with respect to $s$. As, by definition, $U_s$ is negative valued this is not compatible with the fact that $\lim_{s\rightarrow+\infty}U_s(x)=0$. Hence, indeed $g(x) = e^{-x}$.
\end{description}
\end{proof}
For a general Gamma random variable, the following result is immediate.
\begin{corollary}
Assume $X$ is a random variable with distribution $\Gamma(\alpha,\theta)$, where $\alpha, \theta>0$. Then, for each $x\geq 0$ fixed,
$$
\lim_{s\rightarrow+\infty}\TXs{s}(x)= \lim_{s\rightarrow+\infty}\TYs{s}\left(\frac{x}{\theta}\right)=e^{-\frac{x}{\theta}},
$$
where $Y$ ie a random variable with distribution $\Gamma(\alpha,1)$.
\end{corollary}
\begin{remark}
The characterization of the limits as the iteration step goes to infinity provides an approximation for the moments of the residual lifetime. Indeed, assuming that $X$ has distribution $\Gamma(\alpha,\theta)$ with $\alpha,\theta>0$, and by recalling (\ref{eq:simpleTs}), 
it follows from the previous discussion that, for each $x\geq 0$ fixed,
$$
\lim_{s\rightarrow+\infty}\frac{1}{\dE X^{s-1}}\dE(X-x)_+^{s-1}=e^{-\frac{x}{\theta}}.
$$
Hence, for $s$ large enough, an approximation for the higher order stop-loss transform $\dE(X-x)_+^{s-1}$, is obtained by considering $e^{-\frac{x}{\theta}}\dE X^{s-1}=e^{-\frac{x}{\theta}}\theta^{s-1}\frac{\Gamma(\alpha+s-1)}{\Gamma(\alpha)}$.
\end{remark}

\section{Iterated distributions induced by Weibull variables}
The Weibull family of distributions is also an important class in reliability theory and lifetime models. Although an explicit and closed form representation for the iterated distributions induced from this family does not seem possible, we may identify the limit behavior as the iteration step goes to $+\infty$. Recall the following representation, mentioned in the proof of Theorem~27 in the corrigendum for \cite{AO18}.
\begin{proposition}
\label{prop1}
Let $X$ be a random variable with absolutely continuous distribution with density $f_X$ and distribution function $F_X$. 
Then, for $s\geq 1$,
$$
\TXs{s}(x)=\int_x^\infty\!\int_{x_1}^\infty\!\cdots\!\int_{x_{k-1}}^\infty H_k(x_k)\,dx_{k}\cdots dx_2 dx_1,
$$
where
$$
H_k(x)=\frac{1}{\prod_{j=1}^{k}\muXs{s-j}} \TXs{s-k}(x).
$$
\end{proposition}
\begin{proof}
Replace successively the representation of each $\TXs{s}$ as an integral, as given by the definition of the iterated distributions, and the result follows immediately.
\end{proof}
We also quote here the characterization corrersponding to Lemma~\ref{lem:thm12} about Weibull distributions.
\begin{lemma}[Theorem~10 in Arab and Oliveira~\cite{AO18}]
\label{lem:thm10}
Let $X$ be a random variable with Weibull distribution with shape parameter $\alpha>0$, and $s\geq 1$ an integer. If $\alpha>1$ (resp. $\alpha<1$), then the failure rate of the $s-$iterated distribution induced by $X$ is increasing (resp. decreasing).
\end{lemma}
\begin{theorem}
\label{th:Weib}
Let $X$ be Weibull distributed with shape parameter $\alpha>0$. Then, for every $x>0$ fixed,
\[
\lim_{s\rightarrow+\infty}\TXs{s}(x) = \begin{cases}
0 & \mbox{if }\alpha>1,\\
e^{-x} & \mbox{if }\alpha=1, \\
1 & \mbox{if }\alpha<1.
\end{cases}
\]
\end{theorem}
\begin{proof}
We need to separate the proof into three different cases, depending on the value for $\alpha$. As mentioned before, it is enough to treat the case where the hazard rate is 1.

\begin{description}
\item[{\rm\textit{The case $\alpha=1$}.}] This is obvious, as this Weibull distribution becomes the exponential distribution.

\item[{\rm\textit{The case $\alpha >1$}.}] By taking $k=s-1$ in the representation given in Proposition~\ref{prop1}, we get
\begin{equation}
\label{eq:rep}
\TXs{s}(x)=\frac1{\dE X^{s-1}}
\int_x^\infty\!\int_{x_1}^\infty\!\cdots\!\int_{x_{s-2}}^\infty\overline{F}_X(x_{s-1})\,dx_{s-1}\cdots dx_2 dx_1.
\end{equation}
We may bound the inner integral above:
\begin{eqnarray*}
\lefteqn{
\int_{x_{s-2}}^\infty \overline{F}_X(x_{s-1})\,dx_{s-1} =
  \int_{x_{s-2}}^\infty e^{-x_{s-1}^\alpha}\,dx_{s-1}} \\
  & &
   =  \frac{e^{-x_{s-2}^\alpha}}{\alpha x_{s-2}^{\alpha-1}}-\frac{\alpha-1}{\alpha}\int_{x_{s-2}}^\infty \frac{e^{-x_{s-1}^\alpha}}{x_{s-1}^\alpha}\,dx_{s-1}
   \leq \frac{e^{-x_{s-2}^\alpha}}{\alpha x_{s-2}^{\alpha-1}}.
\end{eqnarray*}
The next integral in the representation (\ref{eq:rep}) is
$$
\int_{x_{s-3}}^\infty\int_{x_{s-2}}^\infty \overline{F}_X(x_{s-1})\,dx_{s-1}dx_{s-2}
\leq
\int_{x_{s-3}}^\infty\frac{e^{-x_{s-2}^\alpha}}{\alpha x_{s-2}^{\alpha-1}}\,dx_{s-2}.
$$
Multiplying and dividing this integrand by $\alpha x_{s-2}^{\alpha-1}$, and integrating by parts, we find the upper bound
$$
\int_{x_{s-3}}^\infty\int_{x_{s-2}}^\infty \overline{F}_X(x_{s-1})\,dx_{s-1}dx_{s-2}
\leq
\frac{e^{-x_{s-3}^\alpha}}{\alpha^2 x_{s-3}^{2(\alpha-1)}}.
$$
Iterating now this argument, it follows that
$$
\TXs{s}(x)\leq \frac1{\dE X^{s-1}} \frac{e^{-x^\alpha}}{\alpha^{s-1} x^{(s-1)(\alpha-1)}}.
$$
Remember that $\dE X^{s-1}=\Gamma(1+\frac{s-1}{\alpha})$, and using Stirling approximation for Gamma function, $\lim_{x\rightarrow+\infty}\frac{\Gamma(1+x)}{\sqrt{2\pi x}\left(\frac{x}{e}\right)^x} =1$, we get
\begin{eqnarray*}
\dE X^{s-1}\alpha^{s-1} x^{(s-1)(\alpha-1)}&\approx& \sqrt{2\pi\frac{s-1}{\alpha}}\left( \frac{s-1}{e\alpha}\right)^{\frac{s-1}{\alpha}}\alpha^{s-1} x^{(s-1)(\alpha-1)}\\
 & \approx & \sqrt{2\pi\frac{s-1}{\alpha}}\left(\left( \frac{s-1}{e\alpha}\right)^{\frac{1}{\alpha}}\alpha x^{\alpha-1}\right)^{s-1},
\end{eqnarray*}
which goes to $+\infty$ when $s\longrightarrow+\infty$, consequently, $\lim_{s\rightarrow+\infty}\TXs{s}(x)=0$.

\item[{\rm\textit{The case $\alpha <1$}.}]
First, we study the sign variation of
\[
V_s(x) = \TXs{s+1}(x) - e^{-\beta x}
 =\frac{1}{\Gamma(1+\frac{s}{\alpha})}\int_{x}^{\infty}(t-x)^{s-1}t^{\alpha-1}e^{-t^\alpha}- e^{-\beta x},
\]
where 
$\beta\geq 0$.
We will describe the sign variation of 
\[
V^\prime_s(x) = \beta e^{-\beta x} - s\frac{\Gamma\left(1+\frac{s-1}{\alpha}\right)}{\Gamma\left(1+\frac{s}{\alpha}\right)}\TXs{s}(x).
\]
If we represent $V_s^\prime$ analogously as done in Proposition~\ref{prop1}, its sign variation may be described from the sign variation of
$$
H_{s-1}(x)
  =\frac{\beta^s}{(s-1)!}e^{-\beta x}-\frac{s}{\Gamma\left(1+\frac{s}{\alpha}\right)}e^{-x^\alpha}
  = \frac{s}{\Gamma\left(1+\frac{s}{\alpha}\right)}e^{-\beta x}
     \left(\frac{\beta^s}{s!}\Gamma\left(1+\frac{s}{\alpha}\right)-e^{\beta x-x^{\alpha}}\right),
$$
which coincides with the sign variation of the large parenthesis above.
%
Since $\alpha<1$ the function $e^{\beta x-x^{\alpha}}$ has a minimum and therefore the possible sign variations for $H_{s-1}$ are ``$-,+,-$'',``$-$'' or ``$+,-$''. Hence the possible sign variations for $V^{\prime}_s$ are ``$-,+,-$'', ``$+,-$'' or ``$-$''. Recalling that $V_s(0)=0$ and $\lim_{x\rightarrow+\infty}V_s(x)=0^{+}$, the latter sign variation is not compatible. In order to decide among the two remaining possible sign variations for $V_s^\prime$, ``$-,+,-$'' or ``$+,-$'', we start by remarking that
\[
V^\prime_s(0) = \beta - \frac{s\Gamma\left(1+\frac{s-1}{\alpha}\right)}{\Gamma\left(1+\frac{s}{\alpha}\right)}=\beta-A(s).
\]
Now, we use the Stirling approximation to describe the asymptotic behavior of $A(s)$, as $s\longrightarrow+\infty$. So, assuming that $s$ is large enough, we have that
\begin{eqnarray*}
\label{stirling}A(s)& \approx &
\frac{s\sqrt{2\pi\frac{s-1}{\alpha}}\left(\frac{s-1}{\alpha e}\right)^{\frac{s-1}{\alpha}}}{\sqrt{2\pi\frac{s}{\alpha}}\left(\frac{s}{\alpha e}\right)^{\frac{s}{\alpha}}} \approx
s\left(\left(1-\frac{1}{s}\right)^{s-1}\frac{e\alpha}{s}\right)^{1/\alpha}\\
 &\approx & s\left(\frac{1}{e}\frac{e\alpha}{s}\right)^{1/\alpha}= \frac{\alpha^{1/\alpha}s}{s^{1/\alpha}},
\end{eqnarray*}
hence $\lim_{s\rightarrow+\infty}A(s)=0$. Moreover, it is easy to verify that $A(s)$ eventually becomes decreasing. Therefore, it follows that, for $s\geq s_0$,  $V^\prime_s(0)\geq 0$, ruling out the sign variation ``$-,+,-$''. 
Thus, the only possible sign variation for $V^\prime_s$ is ``$+,-$'' which implies that $V_s(x)\geq 0$. In other words, $\TXs{s}\geq e^{-\beta x}$ for sufficiently large $s$ and by choosing $\beta$ sufficiently close to zero, this means that $\lim_{s\rightarrow+\infty}\TXs{s}\longrightarrow 1$.
\end{description}
\end{proof}
\begin{remark}
The result of Theorem \ref{th:Weib} provides upper and lower bounds for the higher order stop-loss transform of a random variable $X$ with a Weibull distribution. In the case where the shape parameter is $\alpha<1$, then for sufficiently large $s\geq s_0$ and $\beta_{s_0}>0$ we have that
\[
e^{-\beta_{s_0} x}\Gamma\left(1+\frac{s}{\alpha}\right)\leq \dE(X-x)^s_+\leq \Gamma\left(1+\frac{s}{\alpha}\right),
\]
where the notation $\beta_{s_0}$ highlights the dependence between the two parameters. For the case of $\alpha>1$, the moments of the residual lifetime approach zero as $s\longrightarrow+\infty$.
\end{remark}

\newpage

\end{document}